\documentclass[10pt]{amsart}
\usepackage[T1]{fontenc}
\usepackage[latin5]{inputenc}
\usepackage{float}
\usepackage{amssymb}
\usepackage{amsfonts}
\usepackage[centertags]{amsmath}
\usepackage{amsthm}
\usepackage{newlfont}
 \makeatletter
\theoremstyle{plain}
\newtheorem{thm}{Theorem}[section]
\numberwithin{equation}{section}
\numberwithin{figure}{section}  
\theoremstyle{plain}

\theoremstyle{plain}

\theoremstyle{plain}
\newtheorem{cor}[thm]{Corollary} 
\theoremstyle{plain}

\theoremstyle{plain}
\newtheorem{lem}[thm]{Lemma} 
\theoremstyle{plain}

\begin{document}
\title{$p$-Hahn Sequence Space}
\author{Murat Kiri\c{s}ci}
\address{Department of Mathematical Education, Hasan Ali Y\"{u}cel Education Faculty, Istanbul University, Vefa, 34470, Fatih, Istanbul, Turkey}
\email{mkirisci@hotmail.com, murat.kirisci@istanbul.edu.tr}
\vspace{0.5cm}
\thanks{This work was supported by Scientific Projects Coordination Unit of Istanbul University. Project number 35565.}
\vspace{0.5cm}

\begin{abstract}
The main purpose of the present paper is to introduce the space $h_{p}$ and study of some properties of new sequence space. Also we compute their dual spaces and characterizations of some matrix transformations.
\end{abstract}

\maketitle

\section{Introduction}
By $\omega = \mathbb{C}^{\mathbb{N}}$, we denote the space of all real- or complex-valued sequences, where $\mathbb{C}$ denotes the complex field and $\mathbb{N}=\{0,1,2,\ldots\}$. Each linear subspace of $\omega$ is called a \emph{sequence space}. For $x=(x_{k})\in \omega$, we shall employ the sequence spaces $\ell_{\infty}=\{x: \sup_{k}|x_{k}|<\infty\}$, $c=\{x: \lim_{k}x_{k} ~\textrm{ exists}~\}$, $c_{0}=\{x: \lim_{k}x_{k}=0\}$,
$bs=\{x: \sup_{n}|\sum_{k=1}^{n}x_{k}|<\infty\}$, $cs=\{x: (\sum_{k=1}^{n}x_{k})\in c\}$ and $\ell_{p}=\{x: \sum_{k}|x_{k}|^{p}<\infty, \quad 1\leq p < \infty\}$ which are Banach space with the following norms; $\|x\|_{\ell_{\infty}}=\sup_{k}|x_{k}|$, $\|x\|_{bs}=\|x\|_{cs}=\sup_{n}|\sum_{k=1}^{n}x_{k}|$ and $\|x\|_{\ell_{p}}=\left(\sum_{k}|x_{k}|^{p}\right)^{1/p}$ as usual, respectively. And also
\begin{eqnarray*}
bv^{p}&=&\left\{x=(x_{k})\in \omega: \sum_{k=1}^{\infty}|x_{k}-x_{k-1}|^{p}<\infty \right\},\\
\int \lambda&=&\left\{x=(x_{k})\in \omega: (kx_{k}) \in \lambda \right\}.
\end{eqnarray*}

A sequence, whose $k-th$ term is $x_{k}$, is denoted by $x$ or $(x_{k})$. \emph{A coordinate space} (or \emph{$K-$space}) is a vector space of numerical sequences, where addition and scalar multiplication are defined pointwise. That is, a sequence space $\lambda$ with a linear topology is called a $K$-space provided each of the maps $p_{i}:\lambda\rightarrow \mathbb{C}$ defined by $p_{i}(x)=x_{i}$ is continuous for all $i\in \mathbb{N}$. A $BK-$space is a $K-$space, which is also a Banach space with continuous coordinate functionals $f_{k}(x)=x_{k}$, $(k=1,2,...)$.A $K-$space $\lambda$ is called an \emph{$FK-$space} provided $\lambda$ is a complete linear metric space. An \emph{$FK-$space} whose topology is normable is called a \emph{$BK-$ space}.If a normed sequence space $\lambda$ contains a sequence $(b_{n})$ with the property that for every $x\in \lambda$ there is unique sequence of scalars $(\alpha_{n})$ such that
\begin{eqnarray*}
\lim_{n\rightarrow\infty}\|x-(\alpha_{0}b_{0}+\alpha_{1}b_{1}+...+\alpha_{n}b_{n})\|=0
\end{eqnarray*}
then $(b_{n})$ is called \emph{Schauder basis} (or briefly basis) for $\lambda$. The series $\sum\alpha_{k}b_{k}$ which has the sum $x$ is then called the expansion of $x$ with respect to $(b_{n})$, and written as $x=\sum\alpha_{k}b_{k}$. An \emph{$FK-$space} $\lambda$ is said to have $AK$ property, if $\phi \subset \lambda$ and $\{e^{k}\}$ is a basis for $\lambda$, where $e^{k}$ is a sequence whose only non-zero term is a $1$ in $k^{th}$ place for each $k\in \mathbb{N}$ and $\phi=span\{e^{k}\}$, the set of all finitely non-zero sequences. If $\phi$ is dense in $\lambda$, then $\lambda$ is called an $AD$-space, thus $AK$ implies $AD$.

Let $\lambda$ and $\mu$ be two sequence spaces, and $A=(a_{nk})$ be an infinite matrix of complex numbers $a_{nk}$, where $k,n\in\mathbb{N}$. Then, we say that $A$ defines a matrix mapping from $\lambda$ into $\mu$, and we denote it by writing $A :\lambda \rightarrow \mu$ if for every sequence $x=(x_{k})\in\lambda$. The sequence $Ax=\{(Ax)_{n}\}$, the $A$-transform of $x$, is in $\mu$; where
\begin{eqnarray}\label{equ1}
(Ax)_{n}=\sum_{k}a_{nk}x_{k}~\textrm{ for each }~n\in\mathbb{N}.
\end{eqnarray}For simplicity in notation, here and in what follows, the summation without limits runs from $0$ to $\infty$. By $(\lambda:\mu)$, we denote the class of all matrices $A$ such that $A :\lambda \rightarrow \mu $. Thus, $A\in(\lambda:\mu)$ if and only if the series on the right side of (\ref{equ1}) converges for each $n\in\mathbb{N}$ and each $x\in\lambda$ and we have $Ax =\{(Ax)_{n}\}_{n \in\mathbb{N}}\in\mu$ for all $x\in\lambda$. A sequence $x$ is said to be $A$-summable to $l$ if $Ax$ converges to $l$ which is called the $A$-limit of $x$.\\

The matrix domain $\lambda_{A}$ of an infinite matrix $A$ in a sequence space $\lambda$ is defined by
\begin{eqnarray}\label{eq0}
\lambda_{A}=\{x=(x_{k})\in\omega:Ax\in\lambda\}
\end{eqnarray}which is a sequence space(for several examples of matrix domains, see \cite{Basarkitap} p. 49-176). In \cite{AB2}, Ba\c{s}ar and Altay have defined the sequence space $bv_{p}$ which consists of all sequences such that $\Delta$-transforms of them are in $\ell_{p}$ where $\Delta$ denotes the matrix $\Delta=(\delta_{nk})$
\begin{eqnarray*}
\Delta=\delta_{nk}= \left\{ \begin{array}{ccl}
(-1)^{n-k}&, & \quad (n-1\leq k \leq n)\\
0&, & \quad (0\leq k < n-1 ~\textrm{ or }~ k>n)
\end{array} \right.
\end{eqnarray*}
for all $k,n\in \mathbb{N}$. The space $[\ell(p)]_{A^{u}}=bv(u,p)$ has been studied by Ba\c{s}ar et al. \cite{BMA} where
\begin{eqnarray*}
A^{u}=a_{nk}^{u}= \left\{ \begin{array}{ccl}
(-1)^{n-k}u_{k}&, & \quad (n-1\leq k \leq n)\\
0&, & \quad (0\leq k < n-1 ~\textrm{ or }~ k>n)
\end{array} \right.
\end{eqnarray*}
for all $k,n\in \mathbb{N}$.

In the present paper, we introduce $p-$Hahn sequence space. We investigate its some properties and compute duals of this space and characterized some matrix transformations.

We assume throughout that $p^{-1}+q^{-1}=1$ for $p,q\geq 1$. We denote the collection of all finite subsets of $\mathbb{N}$ be $\mathcal{F}$.

\section{New Hahn Sequence Space}

Hahn \cite{Hahn} introduced the $BK-$space $h$ of all sequences $x=(x_{k})$ such that
\begin{eqnarray*}
h=\left \{x: \sum_{k=1}^{\infty} k|\Delta x_{k}|<\infty  ~\textrm{ and }~  \lim_{k \rightarrow \infty}x_{k}=0 \right \},
\end{eqnarray*}
where $\Delta x_{k}=x_{k}-x_{k+1}$, for all $k\in \mathbb{N}$. The following norm
\begin{eqnarray*}
\|x\|_{h}=\sum_{k} k|\Delta x_{k}|+\sup_{k}|x_{k}|
\end{eqnarray*}
was defined on the space $h$ by Hahn \cite{Hahn} (and also \cite{Goes}). Rao (\cite{Rao}, Proposition 2.1) defined a new norm on $h$ as $\|x\|=\sum_{k} k|\Delta x_{k}|.$ Goes and Goes \cite{Goes} proved that the space $h$ is a $BK-$space.

Hahn proved following properties of the space $h$:
\begin{lem}\label{lem1}
\begin{itemize}
\item[(i)] $h$ is a Banach space.
\item[(ii)]$h \subset \ell_{1} \cap \int c_{0}.$
\item[(iii)] $h^{\beta}=\sigma_{\infty}.$
\end{itemize}
\end{lem}

In \cite{Goes}, Goes and Goes studied functional analytic properties of the $BK-$space $bv_{0}\cap d\ell_{1}$. Additionally, Goes and Goes considered the arithmetic means of sequences in $bv_{0}$ and $bv_{0}\cap d\ell_{1}$,  and used an important fact which the sequence of arithmetic means $(n^{-1}\sum_{k=1}^{n}x_{k})$ of an $x\in bv_{0}$ is a quasiconvex null sequence. And also Goes and Goes proved that $h=\ell_{1}\cap \int bv=\ell_{1}\cap \int bv_{0}$.\\

Rao \cite{Rao} studied some geometric properties of Hahn sequence space and gave the characterizations of some classes of matrix transformations.\\

Balasubramanian and Pandiarani\cite{balasub} defined the new sequence space $h(F)$ called the Hahn sequence space of fuzzy numbers and proved that $\beta-$ and $\gamma-$duals of $h(F)$ is the Ces\`{a}ro space of the set of all fuzzy bounded sequences.\\

Kiri\c{s}ci \cite{kirisci1} compiled to studies on Hahn sequence space and defined a new Hahn sequence space by Ces\`{a}ro mean in \cite{kirisci2}.\\

Now, we introduce the sequence space $h_{p}$ by
\begin{eqnarray*}
h_{p}=\left\{x: \sum_{k=1}^{\infty} \left(k|\Delta x_{k}|\right)^{p}<\infty  ~\textrm{ and }~  \lim_{k \rightarrow \infty}x_{k}=0\right\}  \quad\quad (1< p < \infty)
\end{eqnarray*}
where $\Delta x_{k}=(x_{k}-x_{k+1})$, $(k=1,2,...)$. If we take $p=1$, $h_{p}=h$ which called Hahn sequence spaces.\\

Define the sequence $y=(y_{k})$, which will be frequently used, by the $M$-transform of a sequence $x=(x_{k})$, i.e.,
\begin{eqnarray}\label{eq1}
y_{k}=(Mx)_{k}=k(x_{k}-x_{k+1}).
\end{eqnarray}
where $M=(m_{nk})$ with
\begin{eqnarray}\label{matrix}
m_{nk}= \left\{ \begin{array}{ccl}
n&, & \quad (n=k)\\
-n&, & \quad (n+1=k)\\
0&, & \quad other
\end{array} \right.
\end{eqnarray}
for all $k,n\in \mathbb{N}$.\\

\begin{thm}\label{tem}
$h_{p}=\ell_{p}\cap \int bv^{p}=\ell_{p}\cap \int bv_{0}^{p}$
\end{thm}
\begin{proof}
We consider
\begin{eqnarray*}
k\Delta x_{k}\leq x_{k}+\Delta (kx_{k}).
\end{eqnarray*}
Then, for $x\in \ell_{p}\cap \int bv^{p}$
\begin{eqnarray*}
\sum_{k=1}^{n}k|\Delta x_{k}|\leq\sum_{k=1}^{n}|x_{k}|+\sum_{k=1}^{n}|\Delta (kx_{k})|
\end{eqnarray*}
and from $|a+b|^{p}\leq 2^{p}\left(|a|^{p}+|b|^{p}\right), (1\leq p <\infty)$, we obtain
\begin{eqnarray*}
\sum_{k=1}^{n}k^{p}|\Delta x_{k}|^{p}\leq2^{p}\left[\sum_{k=1}^{n}|x_{k}|^{p}+\sum_{k=1}^{n}|\Delta (kx_{k})|^{p}\right].
\end{eqnarray*}
For each positive integer $r$, we get
\begin{eqnarray*}
\sum_{k=1}^{r}k^{p}|\Delta x_{k}|^{p}\leq2^{p}\left[\sum_{k=1}^{r}|x_{k}|^{p}+\sum_{k=1}^{r}|\Delta (kx_{k})|^{p}\right].
\end{eqnarray*}
and as $r\rightarrow\infty$
\begin{eqnarray*}
\sum_{k=1}^{\infty}k^{p}|\Delta x_{k}|^{p}\leq2^{p}\left[\sum_{k=1}^{\infty}|x_{k}|^{p}+\sum_{k=1}^{\infty}|\Delta (kx_{k})|^{p}\right].
\end{eqnarray*}
and $\lim_{k\rightarrow\infty}x_{k}=0$. Then $x\in h_{p}$ and
\begin{eqnarray}\label{inc1}
\ell_{p}\cap \int bv^{p}\subset h_{p}.
\end{eqnarray}

Let $x \in h_{p}$ and we consider
\begin{eqnarray*}
\sum_{k=1}^{\infty}|x_{k+1}|^{p}-\sum_{k=1}^{\infty}|\Delta (kx_{k})|^{p}\leq\sum_{k=1}^{\infty}k^{p}|\Delta x_{k}|^{p}.
\end{eqnarray*}
The the series $\sum_{k=1}^{\infty}|x_{k+1}|^{p}$ is convergent from the definition of $\ell_{p}$. Also
$\sum_{k=1}^{\infty}|\Delta (kx_{k})|^{p}<\infty$ and therefore $x\in \ell_{p}\cap \int bv^{p}$. Then
\begin{eqnarray}\label{inc2}
h_{p}\subset\ell_{p}\cap \int bv^{p}.
\end{eqnarray}

Form (\ref{inc1}) and (\ref{inc2}), we obtain $h_{p}=\ell_{p}\cap \int bv^{p}$.
\end{proof}

\begin{thm}\label{thmBK}
The sequence space $h_{p}$ is a $BK$-space with $AK$.
\end{thm}

\begin{proof}
If $x$ is any sequence, we write $\sigma_{n}(x)=M_{n}x$. Let $\varepsilon >0$ and $x\in h_{p}$ be given. Then there exists $N$ such that
\begin{eqnarray}\label{AK1}
|\sigma_{n}(x)|<\varepsilon / 2
\end{eqnarray}
for all $n\geq N$. Now let $m\geq N$ be given. Then we have for all $n\geq m+1$ by (\ref{AK1})
\begin{eqnarray*}
\left|\sigma_{n}\left(x-x^{[m]}\right)\right|\leq \left[\sum_{k=m+1}^{\infty}\bigg|k(\Delta x_{k})\bigg|^{p}\right]^{1/p}\leq |\sigma_{n}(x)|+|\sigma_{m}(x)|<\varepsilon / 2 +\varepsilon / 2 =\varepsilon
\end{eqnarray*}
whence $\|x-x^{[m]}\|_{h_{p}}\leq \varepsilon$ for all $m\geq N$. This shows $x=\lim_{m\rightarrow\infty}x^{[m]}$.
\end{proof}

Since $h_{p}$ is an $AK$-space and every $AK$-space is $AD$, we can give the following corollary:

\begin{cor}\label{corAD}
The sequence space $h_{p}$ has $AD$.
\end{cor}

\begin{thm}
Define a sequence $b^{(k)}=\big\{ b_{n}^{(k)} \big\}_{n\in \mathbb{N}}$ of elements of the space $h_{p}$ for every fixed $k\in \mathbb{N}$ by
\begin{eqnarray*}
b_{n}^{(k)}= \left\{ \begin{array}{ccl}
\frac{1}{k}&, & \quad (n\leq k)\\
0&, & \quad (n> k)
\end{array} \right.
\end{eqnarray*}
Then the sequence $\bigl\lbrace b_{n}^{(k)}\bigr\rbrace_{n\in \mathbb{N}}$ is a basis for the space $h_{p}$, and any $x\in h_{p}$ has a unique representation of the form
\begin{eqnarray}\label{eq20}
x=\sum_{k}\lambda_{k}b^{(k)}
\end{eqnarray}
where $\lambda_{k}=(Mx)_{k}$ for all $k\in \mathbb{N}$ and $1\leq p <\infty$.
\end{thm}

\begin{proof}
It is clear that $\{b^{(k)}\}\subset h_{p}$, since
\begin{eqnarray}\label{eq18}
Mb^{(k)}=e^{k}\in \ell_{1}, \quad (k=0,1,2,...).
\end{eqnarray}
$1\leq p <\infty$. Let $x\in h_{p}$ be given. For every non-negative integer $m$, we put
\begin{eqnarray}\label{eq19}
x^{[m]}=\sum_{k=0}^{m}\lambda_{k}b^{(k)}.
\end{eqnarray}
Then, we obtain by applying $M$ to (\ref{eq19}) with (\ref{eq18}) that
\begin{eqnarray*}
Mx^{[m]}=\sum_{k=0}^{m}\lambda_{k}Mb^{(k)}=\sum_{k=0}^{m}(Mx)_{k}e^{k}
\end{eqnarray*}
and
\begin{eqnarray*}
\left\{M(x-x^{[m])}\right\}_{i}= \left\{ \begin{array}{ccl}
0&, & \quad (0\leq i \leq m)\\
(Mx)_{i}&, & \quad (i>m)
\end{array}; \quad \quad (i,m\in \mathbb{N}). \right.
\end{eqnarray*}
Given $\varepsilon >0$, then there is an integer $m_{0}$ such that
\begin{eqnarray*}
\left[\sum_{i=m}^{\infty}|i.(\Delta x)_{i}|^{p}\right]^{1/p}<\frac{\varepsilon}{2}
\end{eqnarray*}
for all $m\geq m_{0}$. Hence,
\begin{eqnarray*}
\|x-x^{[m]}\|_{h_{p}}=\left[\sum_{i=m}^{\infty}|i.(\Delta x)_{i}|^{p}\right]^{1/p}\leq \left[\sum_{i=m_{0}}^{\infty}|i.(\Delta x)_{i}|^{p}\right]^{1/p}< \frac{\varepsilon}{2} < \varepsilon
\end{eqnarray*}
for all $m\geq m_{0}$ which proves that $x\in h_{p}$ is represented as in (\ref{eq20}).\\

To show the uniqueness of this representation, we assume that $x=\sum_{k}\mu_{k}b^{(k)}$. Now, we define the transformation $T$ with the notation of (\ref{eq1}), from $h_{p}$ to $\ell_{p}$ by $x \mapsto y=Tx$. The linearity of $T$ is clear. Since the linear transformation $T$ is continuous we have at this stage that
\begin{eqnarray*}
(Mx)_{n}=\sum_{k}\mu_{k}\{Mb^{(k)}\}_{n}=\sum_{k}\mu_{k}e_{n}^{k}=\mu_{n}; \quad (n\in\mathbb{N})
\end{eqnarray*}
which contradicts the fact that $(Mx)_{n}=\lambda_{n}$ for all $n\in\mathbb{N}$. Hence, the representation (\ref{eq20}) of $x\in h_{p}$ is unique.
\end{proof}

\begin{thm}
Except the case $p=2$, the space $h_{p}$ is not an inner product space, therefore not a Hilbert space for $1< p < \infty$.
\end{thm}

\begin{proof}
For $p=2$, we will show that the space $h_{2}$ is a Hilbert space. Since the space $h_{p}$ is a $BK$-space from Theorem \ref{thmBK}, the space $h_{2}$ is  a $BK$-space, for $p=2$. Also its norm can be obtained from an inner product, i.e., $\|x\|_{h_{2}}=\langle k\Delta x, k\Delta x\rangle^{1/2}$ holds. Then the space $h_{2}$ is a Hilbert space.\\

Now consider the sequences $e_{1}=(1,0,0,0,\cdots)$ and $e_{2}=(0,1,0,0,\cdots)$. Then we see that $\|e_{1}+e_{2}\|_{h_{p}}^{2}+\|e_{1}-e_{2}\|_{h_{p}}^{2}\neq 2.\big(\|e_{1}\|_{h_{p}}^{2}+\|e_{2}\|_{h_{p}}^{2}\big)$, i.e., the norm of the space $h_{p}$ does not satisfy the parallelogram equality, which menas that the norm cannot be obtained from inner product. Hence, the space $h_{p}$ with $p\neq 2$ is a Banach space that is not a Hilbert space.
\end{proof}

Now, we give some inclusion relations concerning with the space $h_{p}$.\\

\begin{thm}
Neither of the spaces $h_{p}$ and $\ell_{\infty}$ includes the other one, where$1< p < \infty$.
\end{thm}

\begin{proof}
Now we choose the sequences $a=(a_{k})$ and $b=(b_{k})$ such that $a=(a_{k})=\{(-1)^{k}\}$ and $b=(b_{k})=\sum_{i=1}^{k}1/(i+1)$. The sequence $a=(a_{k})$ is in $\ell_{\infty} \backslash h_{p}$ and the sequence $b=(b_{k})$ is in $h_{p} \backslash \ell_{\infty}$. So, the sequences $h_{p}$ and $\ell_{\infty}$ does not include each other.
\end{proof}

\begin{thm}
If $1\leq p <r$, then $h_{p} \subset h_{r}$.
\end{thm}
\begin{proof}
This can be obtained by analogy with the proof of Theorem 2.6 in \cite{AB2}. So, we omit the details.
\end{proof}

\section{Duals of New Hahn Sequence Space}
In this section, we state and prove the theorems determining the $\alpha$-, $\beta$- and $\gamma$-duals of the sequence space $h_{p}$.

Let $x$ and $y$ be sequences, $X$ and $Y$ be subsets of $\omega$ and $A=(a_{nk})_{n,k=0}^{\infty}$ be an infinite matrix of complex numbers. We write $xy=(x_{k}y_{k})_{k=0}^{\infty}$, $x^{-1}*Y=\{a\in\omega: ax\in Y\}$ and $M(X,Y)=\bigcap_{x\in X}x^{-1}*Y=\{a\in\omega: ax\in Y ~\textrm{ for all }~ x\in X\}$ for the \emph{multiplier space} of $X$ and $Y$. In the special cases of $Y=\{\ell_{1}, cs, bs\}$, we write $x^{\alpha}=x^{-1}*\ell_{1}$, $x^{\beta}=x^{-1}*cs$, $x^{\gamma}=x^{-1}*bs$ and $X^{\alpha}=M(X,\ell_{1})$, $X^{\beta}=M(X,cs)$, $X^{\gamma}=M(X,bs)$ for the $\alpha-$dual, $\beta-$dual, $\gamma-$dual of $X$. By $A_{n}=(a_{nk})_{k=0}^{\infty}$ we denote the sequence in the $n-$th row of $A$, and we write $A_{n}(x)=\sum_{k=0}^{\infty}a_{nk}x_{k}$ $n=(0,1,...)$ and $A(x)=(A_{n}(x))_{n=0}^{\infty}$, provided $A_{n}\in x^{\beta}$ for all $n$.\\

Given an $FK-$space $X$ containing $\phi$, its conjugate is denoted by $X'$ and its $f-$dual or sequential dual is denoted by $X^{f}$ and is given by $X^{f}=\{$ all sequences $(f(e^{k})): f \in X'\}$.\\

Let $\lambda$ be a sequence space. Then $\lambda$ is called \emph{perfect} if $\lambda=\lambda^{\alpha\alpha}$; \emph{normal} if $y\in\lambda$ whenever $|y_{k}|\leq |x_{k}|, \quad k\geq 1$ for some $x\in\lambda$; monotone if $\lambda$ contains the canonical preimages of all its stepspace.

\begin{lem}\label{lemmaR1}
\begin{itemize}
\item [(i).] $A\in (h:\ell_{1})$ if and only if
\begin{eqnarray}\label{eq30}
\sum_{n=1}^{\infty}|a_{nk}|  ~\textrm{ converges, }~   (k=1,2,...)
\end{eqnarray}
\begin{eqnarray}\label{eq40}
\sup_{k}\frac{1}{k}\sum_{n=1}^{\infty}\bigg|\sum_{\upsilon=1}^{k}a_{n\upsilon}\bigg|<\infty.
\end{eqnarray}
\item[(ii).] $A\in (\ell_{p}:\ell_{1})$ if and only if
\begin{eqnarray*}
\sup_{K\in \mathcal{F}}\sum_{k}\bigg|\sum_{n\in K}a_{nk}\bigg|^{q}<\infty
\end{eqnarray*}
\end{itemize}
\end{lem}

\begin{lem}\label{lemmaR2}
\begin{itemize}
\item [(i).] $A\in (h:c)$ if and only if
\begin{eqnarray}\label{eq50}
\sup_{n,k}\frac{1}{k}\bigg|\sum_{\upsilon=1}^{k}a_{n\upsilon}\bigg|<\infty
\end{eqnarray}
\begin{eqnarray}\label{eq60}
\lim_{n\rightarrow \infty}a_{nk}  ~\textrm{ exists, }~   (k=1,2,...)
\end{eqnarray}
\item [(ii).] $A\in (\ell_{p}:c)$ if and only if (\ref{eq60}) holds and
\begin{eqnarray}\label{eq500}
\sup_{n}\sum_{k}\big|a_{nk}\big|^{q}<\infty, \quad 1<p<\infty
\end{eqnarray}
\end{itemize}
\end{lem}

\begin{lem}\label{lemmaR4}
\begin{itemize}
\item[(i).] $A\in (h:\ell_{\infty})$ if and only if (\ref{eq50}) holds.
\item[(ii).] $A\in (\ell_{p}:\ell_{\infty})$ if and only if (\ref{eq500}) holds with $1<p\leq\infty$.
\end{itemize}
\end{lem}

\begin{lem}\label{lemmaR3}
$A\in (h:c_{0})$ if and only if (\ref{eq50}) holds and
\begin{eqnarray}\label{eq70}
\lim_{n\rightarrow\infty}a_{nk}=0
\end{eqnarray}
\end{lem}

\begin{lem}\label{lemmaR5}
$A\in (h:h)$ if and only if (\ref{eq70}) holds and
\begin{eqnarray}\label{eq80}
\sum_{n=1}^{\infty}n|a_{nk}-a_{n+1,k}| ~\textrm{ converges, }~ (k=1,2,...)
\end{eqnarray}
\begin{eqnarray}\label{eq90}
\sup_{k}\frac{1}{k}\sum_{n=1}^{\infty}n\bigg|\sum_{v=1}^{k}(a_{nv}-a_{n+1,v})\bigg|<\infty.
\end{eqnarray}
\end{lem}

\begin{thm}\label{thmalpha}
We define the sets $d_{1}$ and $d_{2}$ as follows:
\begin{eqnarray*}
d_{1}&=&\{a=(a_{k})\in \omega: \sup_{K\in\mathcal{F}}\sum_{k}\left|\sum_{n\in K}\frac{1}{k}a_{n}\right|^{q}<\infty\} \quad 1<p<\infty\\
d_{2}&=&\{a=(a_{k})\in \omega:\sup_{K\in\mathcal{F}}\sum_{k}\left|\sum_{n\in K}\frac{1}{k}a_{n}\right|<\infty\}.
\end{eqnarray*}
Then $[h_{p}]^{\alpha}=d_{1}$ and $[h]^{\alpha}=d_{2}$.
\end{thm}

\begin{proof}
We give the proof only for the case $[h_{p}]^{\alpha}=d_{1}$. Let us take any $a=(a_{k})\in \omega$ and consider the equation
\begin{eqnarray}\label{eq100}
a_{n}x_{n}=\sum_{j=n}^{\infty}\frac{a_{n}}{j}y_{j}=(Dy)_{n} \quad (n\in \mathbb{N})
\end{eqnarray}
where $D=(d_{nk})$ is defined by
\begin{eqnarray*}
d_{nk}= \left\{ \begin{array}{ccl}
\frac{a_{n}}{k}&, & \quad k\geq n\\
0&, & \quad k<n
\end{array} \right.
\end{eqnarray*}
for all $k,n\in \mathbb{N}$. It follows from (\ref{eq100}) with Lemma \ref{lemmaR1}(ii) that  $ax=(a_{n}x_{n})\in \ell_{1}$ whenever $x=(x_{k})\in h_{p}$ if and only if $Dy\in \ell_{1}$ whenever $y=(y_{k})\in \ell_{p}$. This means that $a=(a_{n})\in [h_{p}]^{\alpha}$ whenever $x=(x_{n})\in h_{p}$ if and only if $D\in (h_{p}:\ell_{1})$. This gives the result that $[h_{p}]^{\alpha}=d_{1}$.
\end{proof}

Hahn\cite{Hahn} proved that $[h]^{\beta}=\sigma_{\infty}$ where $\sigma_{\infty}=\{a=(a_{k})\in \omega:\sup_{n}\frac{1}{n}|\sum_{k=1}^{n}a_{k}|<\infty\}$. We can give $\beta$-dual of $h_{p}$.

\begin{thm}\label{thmbeta}
Let $1<p<\infty$. Then, $[h_{p}]^{\beta}=d_{3}$
where
\begin{eqnarray*}
d_{3}=\left\{a=(a_{k})\in \omega: \sup_{n\in\mathbb{N}}(n^{-1})^{q}\sum_{k}\left|\sum_{j=k}^{n}a_{j}\right|^{q}<\infty\right\}
\end{eqnarray*}
\end{thm}

\begin{proof}
Consider the equation
\begin{eqnarray}\label{eqbeta}
\sum_{k=1}^{n}a_{k}x_{k}=\sum_{k=1}^{n}a_{k}\left(\sum_{j=k}^{n}\frac{y_{j}}{j}\right)=\sum_{k=1}^{n}\left(\sum_{j=1}^{k}\frac{a_{j}}{k}\right)y_{k}=(By)_{n}  \quad \quad (n\in \mathbb{N});
\end{eqnarray}
where $B=(b_{nk})$ are defined by
\begin{eqnarray*}
b_{nk}= \left\{ \begin{array}{ccl}
\sum_{j=1}^{k}\frac{a_{j}}{k}&, & \quad (n\leq k)\\
0&, & \quad (n>k)
\end{array} \right.
\end{eqnarray*}
for all $k,n\in \mathbb{N}$. Thus we deduce from Lemma \ref{lemmaR2} (ii) with (\ref{eqbeta}) that $ax=(a_{k}x_{k})\in cs$ whenever $x=(x_{k})\in h_{p}$ if and only if $By\in c$ whenever $y=(y_{k})\in \ell_{p}$. Thus, $(a_{k})\in cs$ and $(a_{k})\in d_{3}$ by (\ref{eq60}) and (\ref{eq500}), respectively. Nevertheless, the inclusion $d_{3}\subset cs$ holds and, thus, we have $(a_{k})\in d_{3}$ whence $[h_{p}]^{\beta}=d_{3}$.
\end{proof}

\begin{lem}\label{lem0}(\cite{Wil84}, Theorem 7.2.7)
Let $X$ be an $FK-$space with  $X\supset \phi$. Then,
\begin{itemize}
\item[(i)] $X^{\beta}\subset X^{\gamma}\subset X^{f}$;
\item[(ii)]If $X$ has $AK$, $X^{\beta}=X^{f}$;
\item[(iii)] If $X$ has $AD$, $X^{\beta}=X^{\gamma}$.
\end{itemize}
\end{lem}

From Theorem \ref{thmBK}, Corollary \ref{corAD} and Lemma \ref{lem0}, we can write the following corollary:

\begin{cor}
\begin{itemize}
\item[(i)] $[h_{p}]^{\beta}=[h_{p}]^{f}$
\item[(ii)] $[h_{p}]^{\beta}=[h_{p}]^{\gamma}$.
\end{itemize}
\end{cor}

\begin{lem}\label{lemperf}
Let $\lambda$ be a sequence space. Then the following assertions are true:
\begin {itemize}
\item[(i)] $\lambda$ is perfect $\Rightarrow$ $\lambda$ is normal $\Rightarrow$ $\lambda$ is monotone;
\item[(ii)] $\lambda$ is normal $\Rightarrow$ $\lambda^{\alpha}=\lambda^{\gamma}$;
\item[(iii)]$\lambda$ is monotone $\Rightarrow$ $\lambda^{\alpha}=\lambda^{\beta}$.
\end{itemize}
\end{lem}

Combining Theorem \ref{thmalpha}, Theorem \ref{thmbeta} and Lemma \ref{lemperf}, we can give the following corollary:

\begin{cor}
The space $h_{p}$ is not monotone and so it is neither normal nor perfect.
\end{cor}

\section{Matrix Transformations}

In this section, we characterize some matrix transformations on the space $h_{p}$.\\

\begin{lem}\label{lemma1}\cite{AB2}
Let $\lambda, \mu$ be any two sequence spaces, $A$ be an infinite matrix and $U$ a triangle matrix matrix.Then, $A\in (\lambda: \mu_{U})$ if and only if $UA\in (\lambda:\mu)$.
\end{lem}

If we define $\widetilde{a}_{nk}=n(a_{nk}-a_{n+1,k})$, then we can give following corollary from Lemma \ref{lemma1} with $U=M$ defined by (\ref{matrix}):
\begin{cor}\label{4cor1}
\begin{itemize}
\item[(i)] $A\in (\ell_{1}:h)$ if and only if
\begin{eqnarray*}
\sup_{k}\sum_{n}\left|\widetilde{a}_{nk}\right|<\infty
\end{eqnarray*}
\item[(ii)] $A\in (c:h)=(c_{0}:h)=(\ell_{\infty}:h)$ if and only if
\begin{eqnarray*}
\sup_{K\in\mathcal{F}}\sum_{n}\left|\sum_{k\in K}\widetilde{a}_{nk}\right|<\infty
\end{eqnarray*}
\end{itemize}
\end{cor}

\begin{thm}\label{4thm1}
Suppose that the entries of the infinite matrices $A=(a_{nk})$ and $E=(e_{nk})$ are connected with the relation
\begin{eqnarray}\label{4e1}
e_{nk}=\overline{a}_{nk}
\end{eqnarray}
for all $k,n\in \mathbb{N}$, where $\overline{a}_{nk}=\sum_{j=k}^{\infty}\frac{a_{nj}}{j}$ and $\mu$ be any sequence space. Then $A\in (h_{p}:\mu)$ if and only if $\{a_{nk}\}_{k\in \mathbb{N}}\in [h_{p}]^{\beta}$ for all $n\in \mathbb{N}$ and $E\in (h:\mu)$.
\end{thm}

\begin{proof}
Let $\mu$ be any given sequence spaces. Suppose that (\ref{4e1}) holds between $A=(a_{nk})$ and $E=(e_{nk})$, and take into account that the spaces $h_{p}$ and $h$ are norm isomorphic.\\

Let $A\in (h_{p}:\mu)$ and take any $y=y_{k}\in h$. Then $EM$ exists and $\{a_{nk}\}_{k\in \mathbb{N}}\in [h_{p}]^{\beta}$ which yields that $\{e_{nk}\}_{k\in \mathbb{N}}\in \ell_{1}$ for each $n\in \mathbb{N}$. Hence, $Ey$ exists and thus
\begin{eqnarray*}
\sum_{k}e_{nk}y_{k}=\sum_{k}a_{nk}x_{k}
\end{eqnarray*}
for all $n\in \mathbb{N}$. We have that $Ey=Ax$ which leads us to the consequence $E\in (h:\mu)$.\\

Conversely, let $\{a_{nk}\}_{k\in \mathbb{N}}\in d_{1}$ for all $n\in \mathbb{N}$ and $E\in (h:\mu)$ hold, and take any $x=x_{k}\in h_{p}$. Then, $Ax$ exists. Therefore, we obtain from the equality
\begin{eqnarray*}
\sum_{k}a_{nk}x_{k}=\sum_{k}\left[\sum_{j=k}^{\infty}\frac{a_{nj}}{j}\right]y_{k}
\end{eqnarray*}
for all $n\in \mathbb{N}$. Thus $Ax=Ey$ and this shows that $A\in (h_{p}:\mu)$.
\end{proof}

If we use the Corollary \ref{4cor1} and change the roles of the spaces $h_{p}$ with $\mu$ in Theorem ref{4thm1}, we can give following theorem:

\begin{thm}\label{4thm2}
Suppose that the entries of the infinite matrices $A=(a_{nk})$ and $\widetilde{A}=(\widetilde{a}_{nk})$ are connected with the relation $\widetilde{a}_{nk}=n(a_{nk}-a_{n+1,k})$ for all $k,n\in \mathbb{N}$ and $\mu$ be any sequence space. Then $A\in (\mu:h_{p})$ if and only if and $\widetilde{A}\in (\mu:h)$.
\end{thm}

\begin{proof}
Let $z=(z_{k})\in \mu$ and consider the following equality
\begin{eqnarray*}
\sum_{k=0}^{m}\widetilde{a}_{nk}z_{k}=\sum_{k=0}^{m}n(a_{nk}-a_{n+1,k})z_{k} \quad ~\textrm{ for all, }~   m,n\in \mathbb{N}
\end{eqnarray*}
which yields that as $m\rightarrow \infty$ that  $(\widetilde{A}z)_{n}=\{M(Az)\}_{n}$ for all $n\in \mathbb{N}$. Therefore, one can observe from here that $Az\in h_{p}$ whenever $z\in \mu$ if and only if $\widetilde{A}z\in h$ whenever $z\in \mu$.
\end{proof}

We can give following corollaries from Lemma \ref{lemmaR1}-\ref{lemmaR5}, Corollary \ref{4cor1}, Theorem \ref{4thm1} and Theorem \ref{4thm2}:

\begin{cor}
\begin{itemize}
\item[(i)] $A\in (h_{p}:\ell_{\infty})$ if and only if $\{a_{nk}\}_{k\in \mathbb{N}}\in [h_{p}]^{\beta}$ for all $n\in \mathbb{N}$ and
\begin{eqnarray}\label{4eq1}
\sup_{k}\left(\frac{1}{k}\left|\sum_{v=1}^{k}\overline{a}_{nv}\right|\right)^{q}<\infty
\end{eqnarray}
\item[(ii)] $A\in (h_{p}:c)$ if and only if $\{a_{nk}\}_{k\in \mathbb{N}}\in [h_{p}]^{\beta}$, (\ref{4eq1}) holds and
\begin{eqnarray}\label{4eq2}
\lim_{n\rightarrow\infty}\overline{a}_{nk}=\alpha_{k} \quad (k\in \mathbb{N})
\end{eqnarray}
\item[(iii)] $A\in (h_{p}:c_{0})$ if and only if $\{a_{nk}\}_{k\in \mathbb{N}}\in [h_{p}]^{\beta}$, (\ref{4eq1}) holds and (\ref{4eq2}) holds with $\alpha_{k}=0$.
\item[(iv)] $A\in (h_{p}:\ell_{1})$ if and only if $\{a_{nk}\}_{k\in \mathbb{N}}\in [h_{p}]^{\beta}$ and
\begin{eqnarray*}
\sum_{n=1}^{\infty}|\overline{a}_{nk}|^{q}  ~\textrm{ converges, }~   (k=1,2,...)\\
\sup_{k}\frac{1}{k^{q}}\sum_{n=1}^{\infty}\bigg|\sum_{\upsilon=1}^{k}\overline{a}_{n\upsilon}\bigg|^{q}<\infty.
\end{eqnarray*}
\end{itemize}
\end{cor}

\begin{cor}
\begin{itemize}
\item[(i)] $A\in (\ell:h_{p})$ if and only if
\begin{eqnarray*}
\sup_{K\in\mathcal{F}}\sum_{k}\left|\sum_{n\in K}\widetilde{a}_{nk}\right|<\infty
\end{eqnarray*}
\item[(ii)] $A\in (c:h_{p})=(c_{0}:h_{p})=(\ell_{\infty}:h_{p})$ if and only if
\begin{eqnarray*}
\sup_{K\in\mathcal{F}}\sum_{n}\left|\sum_{k\in K}\widetilde{a}_{nk}\right|<\infty
\end{eqnarray*}
\end{itemize}
\end{cor}

\section{Conclusion}
Hahn \cite{Hahn} defined the space $h$ and gave some properties. Goes and Goes \cite{Goes} studied its different properties. Rao \cite{Rao} introduced the Hahn sequence space and investigated some properties in Banach space theory. Kiri\c{s}ci \cite{kirisci1} compiled to studies of Hahn sequence space and defined a new Hahn sequence space by Ces\`{a}ro mean in \cite{kirisci2}.\\

In this paper, we defined the space $p-$Hahn sequence spaces and gave some properties. In section 3, we compute the duals of the space $h_{p}$ and characterize some matrix transformations related to this space, in section 4.

Finally, we should note that, as a natural continuation of the present paper, one can study the paranormed Hahn sequence space. Also it can be obtained the new Hahn sequence space by using Euler mean, Riesz mean, generalized weighted mean etc.

\end{document}